\documentclass[oneside]{amsart}
\usepackage{enumerate,amssymb,mathtools}
\usepackage{musicography}
\usepackage{mathrsfs}
\newtheorem{thm}{Theorem}[section]
\newtheorem{coro}[thm]{Corollary}
\newtheorem{prop}[thm]{Proposition}
\newtheorem{lem}[thm]{Lemma}

\newtheorem{Claim}[thm]{Claim}
\theoremstyle{definition}
\newtheorem{defn}[thm]{Definition}
\newtheorem{question}[thm]{Question}
\newtheorem{questions}[thm]{Questions}
\newcommand{\Rset}{\mathbb{R}}
\newcommand{\Tset}{\mathbb{T}}
\newcommand{\Nset}{\omega}

\newcommand{\Cset}{2^\Nset}

\newcommand{\Pset}{\Nset^\Nset}
\newcommand{\UPset}{\Nset^{\uparrow\Nset}}

\newcommand{\eps}{\varepsilon}
\newcommand{\del}{\delta}
\newcommand{\subs}{\subseteq}

\newcommand{\upto}{{\nearrow}}

\newcommand{\rest}{{\restriction}}
\renewcommand{\leq}{\leqslant}
\renewcommand{\geq}{\geqslant}

\newcommand{\abs}[1]{\lvert#1\rvert}
\newcommand{\seq}[1]{\langle#1\rangle}
\newcommand{\seqeps}{\seq{\eps_n}}

\newcommand{\Seqeps}{\seqeps\in(0,\infty)^\Nset}

\newcommand{\si}{$\sigma$\nobreakdash-}
\newcommand{\hm}{\mathscr H}

\newcommand{\pack}{\mathscr P}
\newcommand{\gbox}{\boldsymbol{\upsilon}}
\newcommand{\ubox}{\gbox}

\newcommand{\cyl}[1]{\seq{#1}}

\DeclareMathOperator{\hdim}{\dim_{\mathsf{H}}}

\DeclareMathOperator{\non}{\mathsf{non}}
\DeclareMathOperator{\add}{\mathsf{add}}

\DeclareMathOperator{\diam}{diam}

\newcommand{\NN}{\mathcal N}

\newcommand{\MM}{\mathcal M}
\newcommand{\EE}{\mathcal E}
\newcommand{\UU}{\mathcal U}

\newcommand{\CC}{\mathcal C}

\newcommand{\addN}{\add\NN}

\newcommand{\mc}[1]{\mathcal{#1}}

\newcommand{\whm}{\boldsymbol{\lambda}}
\newcommand{\ym}{\boldsymbol{\upsilon}}
\newcommand{\grG}{\mathbb{G}}
\newcommand{\grH}{\mathbb{H}}
\DeclareMathOperator{\mult}{\boldsymbol{\mathsf{m}}}
\newcommand{\EEE}{\boldsymbol{\mc E}}
\newcommand{\AAA}{\boldsymbol{\mc A}}
\newcommand{\smz}{\ensuremath{\boldsymbol{\mathsf{Smz}}}}
\newcommand{\ssmz}{\ensuremath{{\boldsymbol{\mathsf{Smz}}^\sharp}}}
\newcommand{\sssmz}{\ensuremath{\boldsymbol{\mathsf{Smz}}^{\musDoubleSharp}}}
\newcommand{\GMS}{\ensuremath{\mathsf{GMS}^{\musDoubleSharp}}}
\newcommand{\wh}[1]{\widehat{#1}}

\newcommand{\fmany}{\forall^\infty}

\newcommand{\NA}{\mathcal{N\mkern-5mu A}}
\newcommand{\MA}{\mathcal{M\mkern-5mu A}}
\newcommand{\EA}{\mathcal{E\mkern-5mu A}}

\newenvironment{enum}{\begin{enumerate}[\rm(i)]}{\end{enumerate}}
\newenvironment{itemyze}%
{\begin{list}{\textbullet}{\setlength{\labelwidth}{1ex}\setlength{\leftmargin}{1.1em}}}%
{\end{list}}
\newenvironment{cproof}{\noindent$\boldsymbol\vdash$}%
{\hfill\qed\par}

\newcommand{\Implies}{\ensuremath{\Rightarrow}}
\begin{document}
\title
[Null-additive sets and packing measures]
{Null-additive sets and packing measures}
\author{Ond\v rej Zindulka}
\address
{Department of Mathematics\\
Faculty of Civil Engineering\\
Czech Technical University\\
Th\'akurova 7\\
160 00 Prague 6\\
Czech Republic}
\email{zindulka@mat.fsv.cvut.cz}
\urladdr{http://mat.fsv.cvut.cz/zindulka}
\thanks{%
The first author gratefully acknowledges support received from a DGAPA-PAPIIT grant IN107526
and a SECIHTI grant CBF-2025-I-898.
}
\begin{abstract}
A separable metric space has strict measure zero 
if its packing measure induced by any gauge is zero.
We provide a few characterizations of strict measure zero, look at cartesian products and 
prove that in every Polish locally compact group every set of strict measure zero 
is null-additive and that in the Cantor set, Euclidean spaces and tori 
the two notions are equivalent.
\end{abstract}
\maketitle
\section{Introduction}
\label{sec:intro}

More than a century ago Borel~\cite{MR1504785} conceived the notion of
\emph{strong measure zero}: by his definition, a metric space $X$ has
\emph{strong measure zero} (hereinafter abbreviated as \smz) if for any sequence $\seq{\eps_n}$
of positive numbers there is a cover $\{U_n\}$ of $X$
such that $\diam U_n\leq\eps_n$ for all $n$.
We mention three different characterizations of the notion:

\textbullet{}
Besicovitch proved in~\cite{MR1555386,MR1555389} that a set has \smz{}
if and only if for every gauge $\phi$ its Hausdorff $\phi$-dimensional measure is zero.

\textbullet{}
Galvin, Mycielski and Solovay~\cite{GMS}, confirming a conjecture of Prikry,
proved a rather surprising characterization of \smz{} subsets of the real line:
A set $X\subs\Rset$ has \smz{} if and only if $X$ is a \emph{Prikry set}, i.e., $X+M\neq\Rset$
for each meager set $M\subs\Rset$.
Their result has been extended to locally compact groups that admit an invariant metric
(\cite{kysiak,MR3723040,wohofskyPhD}) and further investigated
in~\cite{MR3453581,MR3707641,MR4146582}.

\textbullet{}
Galvin~\cite{MR3696064} provided a long-expected game-theoretic characterization of
\smz{} in \si totally bounded spaces.

\medskip
In the papers~\cite{MR3946667,MR4472524} a parallel theory of \emph{sharp measure zero}
 (hereafter abbreviated as \ssmz) was developed:

\textbullet{}
By one of many equivalent definitions, a set in a metric space
has \ssmz{} if for every gage $\phi$
the so called \emph{upper Hausdorff} $\phi$-dimensional \emph{measure} is zero.
The Upper Hausdorff measure is defined just like the Hausdorff measure except that it is defined from
finite covers, instead of countable covers; see~\cite{MR3946667} for details.

\textbullet{}
In locally compact metric groups admitting an invariant metric, \ssmz{}
is equivalent with \emph{meager-additivity}: a set $X\subs\Rset$ is \emph{meager-additive}
if $X+M$ is meager for every meager set $M\subs\Rset$;
and the notion easily extends to topological groups. This is clearly an analogue
of the mentioned Galvin-Mycielski-Solovay Theorem.

\textbullet{}
A characterization of \ssmz{} (and thus meager-additivity) in terms of
a game is provided.

\medskip
So \smz{} and \ssmz{} have strikingly similar characterizations, in particular in terms of
fractal measures. One may, just out of curiosity, ask what results from replacing Hausdorff
or upper Hausdorff measures with some other fractal measures, e.g., packing measures.
So let's say that a set in a metric space has \emph{strict measure zero} (abbreviated as \sssmz)
if for every gauge $\phi$ its packing measure induced by $\phi$ is zero.
To my surprise, it turns out that

\textbullet{}
a set $X\subs\Rset$ has \sssmz{} if and only if it is  \emph{null-additive},
i.e., $X+N$ is Lebesgue null for every Lebesgue null set $N\subs\Rset$,
so there indeed is a Galvin-Mycielski-Solovay Theorem-like characterization of null-additive
sets.
We are actually able to prove this characterization in the Cantor set and Euclidean spaces.
We also show that in any locally compact metric group every \sssmz{} set is null-additive.

\textbullet{}
There is a also characterization of \sssmz{} in terms of a game that parallels that of
Galvin~\cite{MR3696064}

\textbullet{}
and also a characterization of \sssmz{} that follows the pattern of
the original Borel's definition of \smz{}.

\smallskip
Say that a locally compact Polish group $\grG$ is \GMS{} if every null-additive set in $\grG$ is \sssmz{}.
We prove that the Cantor set, the real line, Euclidean spaces and $n$-dimesional tori are \GMS{} groups
and that  every set in a \GMS{} group with an invariant metric satisfies the following diagram.
\begin{center}\renewcommand{\arraystretch}{1.4}
\newcommand{\lra}{\Longrightarrow}
\newcommand{\uda}{\Updownarrow}
\newcommand{\da}{\Downarrow}
\newcommand{\ua}{\Uparrow}
\begin{tabular}{ccccccc}
\sssmz         & $\lra$ & \ssmz           & $\lra$ & \smz\\
$\uda$         &        & $\uda$          &        &$\uda$\\
null-additive  & $\lra$ & meager-additive & $\lra$ & Prikry
\end{tabular}
\end{center}
Note that null-additive \Implies{} meager-additive in the Cantor set is due to Shelah~\cite{MR1324470},
\ssmz{} $\Leftrightarrow$ meager-additive in locally compact groups with an invariant metric to
Zindulka~\cite{MR4472524} and \smz{} $\Leftrightarrow$ Prikry in locally compact groups 
to Kysiak~\cite{MR4472524}, and some implications survive  in a wider scope, 
see~\cite{MR3453581,MR3707641,MR4472524}.

\smallskip
Section~\ref{sec:fractal} recalls Hausdorff and packing measures. 
Section~ref{sec:sssmz} provides the definition of strict measure zero and a few equivalent definitions.
We also show here that a product of strict measure zero sets has strict measure zero.
In Section~\ref{sec:galvin} we prove a game-theoretic characterization of strict measure zero.
In Section~\ref{sec:nadd} we show that in the groups of reals ---
Cantor set, the line and circle --- strict measure zero is equivalent with 
null-additivity, providing thus an analogue to Galvin-Mycielski-Solovay Theorem for null-additive sets
in these groups.
In Section~\ref{sec:nadditive} we prove that the implication strict measure zero \Implies{} null-additive
remains valid in every locally compact Polish group. Then we define a locally compact group to be \GMS{} 
if each its null-additive subset has strict measure zero and notice that the class of \GMS{} groups
is closed under finite Cartesian products. It follows that all Euclidean spaces and tori are \GMS{}.
We also prove that in a \GMS{} group, every null-additive set is meager-additive. 
Another, maybe surprising, result is that if $X\subs\Rset$ is a null-additive set,
then $\hm^s(X+N)=0$ whenever $\hm^s(N)=0$ holds for any $s>0$, so null-additivity 
w.r.t.~Lebesgue measure yields null-additivity w.r.t~all Hausdorff measures; 
and the result extends to all \GMS{} groups and gauges.
In Section~\ref{sec:rem} we look at a few related topics, and ask a few questions. 
The question of major importance is if all ocally compact Polish groups are \GMS{}.

\subsection*{Borel Conjecture}
Borel conjectured that all \smz{} sets are countable.
It is well-known that Borel Conjecture is independent of ZFC, the usual axioms of set theory.
The proof of the consistency of its failure was settled by 1940
by Sierpi\'nski~\cite{sierpinski}, who proved in 1928 that the Continuum Hypothesis yields
an uncountable \smz{} set, and G\"odel~\cite{MR0002514}, who proved in 1940
the consistency of the Continuum Hypothesis.
The consistency of the Borel Conjecture remained open until 1976
when Laver proved in his ground-breaking paper~\cite{MR0422027}
Borel Conjecture to be indeed consistent with ZFC.
Carlson~\cite{MR1139474} wraped up in 1993 by proving that the Borel Conjecture
implies that every separable \smz{} metric space is countable.

So all of the three notions -- \smz, \ssmz{} and \sssmz{} -- consistently represent countable
sets, so the whole discussion is void. On the other hand, as proved in~\cite[Theorem 8]{MR738943}, 
Continuum Hypothesis yields uncountable \emph{strong $\gamma$-sets} that are trivially \sssmz{}, 
cf.~Section~\ref{sec:rem}.

Some common notation used throughout the paper includes $\abs{A}$ for
the cardinality of a set $A$, $\Nset$ for the set of natural numbers,
$[\Nset]^\Nset$ for the collection of infinite
subsets of $\Nset$,
and $\UPset$ for the family of nondecreasing
unbounded sequences of natural numbers.

\section{Hausdorff and packing measure}
\label{sec:fractal}

Let $X$ be a metric space. The diameter of a set $A\subs X$ is denoted by $\diam A$.
A family $\mc E$ of sets is \emph{$\del$-fine} if $\forall E\in\EE\ \diam E\leq\del$.
A closed ball of radius $r$ centered at $x$ is denoted $B(x,r)$.
We also write $B(A,r)=\bigcup_{x\in A}B(x,r)$.

A function $\phi\colon[0,\infty)\to[0,\infty)$ is a \emph{gauge} if it is non-decreasing,
right-continuous and $\phi(r)=0$ iff $r=0$.
The following is the common ordering of gauges, cf.~\cite{MR0281862}:
$$
  \phi\prec\psi\quad \overset{\mathrm{def}}{\equiv}
  \quad\lim_{r\to0+}\frac{\psi(r)}{\phi(r)}=0.
$$
If there is $s>0$ such that $\phi(r)=r^s$, then we identify $\phi$ with $s$.

\subsection*{Hausdorff measure}
Given a gauge $\phi$, we first define for each $\del>0$
\[
  \hm_\del^\phi(E)=\inf\Bigl\{\sum_{n\in\Nset}\phi(\diam E_n):
  \{E_n\}\text{ is a $\del$-fine cover of $E$}\Bigr\}
\]
and then define the \emph{Hausdorff measure induced by $\phi$} to be
$\hm^\phi(E)=\sup_{\del>0}\hm_\del^\phi(E)$.
In the common case when $\phi(r)=r^s$ for some
$s>0$ we write $\hm^s$ for $\hm^\phi$, and the same license is used for the packing
measure and box-counting submeasure and their premeasures defined \emph{infra}.

\subsection*{Packing measure}
Say that a family $\{(x_i,r_i):i\in I\}\subs E\times(0,\infty)$ is a \emph{packing of $E$}
if $x_j\notin B(x_i,r_i)$ for any two distinct $i, j\in I$. The packing is $\del$-fine if
$r_i\leq\del$ for all $i\in I$.
Given a gauge $\phi$, we first define for each $\del>0$
\begin{align*}
  \pack_\del^\phi(E)&=\sup\Bigl\{\sum_{n\in\Nset}\phi(2r_i):
  \{(x_i,r_i):i\in I\}\text{ is a $\del$-fine packing of $E$}\Bigr\},\\
  \pack_0^\phi(E)&=\inf_{\del>0}\pack_\del^\phi(E)
  \shortintertext{and then, since $\pack_0^\phi$ is not a measure,}
  \pack^\phi(E)&=\inf\Bigl\{\sum_{n\in\Nset}\pack_0^\phi(E_n):
    \text{ $\{E_n\}$ is a cover of $E$}\Bigr\}.
\end{align*}
The set function $\pack^\phi$ is called a \emph{packing measure induced by $\phi$}.
It is, just like Hausdorff measure, an outer measure whose restriction to Borel sets
is a Borel measure on $X$.

\subsection*{Box-counting submeasure}
The following notion, while in use, is not as well-known as Hausdorff and packing measures.
It is easier to deal with than packing measure and can often replace it.
The \emph{covering number function} $N_E(\del)$ of $E$ is defined
as the minimal size of the $\del$-fine cover of $E$.
Given a gauge $\phi$, we first define for each $\del>0$
\begin{align*}
  \ubox_\del^\phi(E)&=\phi(\del)N_E(\del),
  \quad\ubox_0^\phi(E)=\limsup_{\del\to0}\ubox_\del^\phi(E)
  \shortintertext{and then}
  \ubox^\phi(E)&=\inf\Bigl\{\sum_{n\in\Nset}\ubox_0^\phi(E_n):
    \text{ $\{E_n\}$ is a cover of $E$}\Bigr\}.
\end{align*}
The set function $\ubox^\phi$ is called a box-counting submeasure induced by $\phi$.
It is \si sub\-additive, but rarely \si additive.
%
Here are some elementary facts about the three fractal (sub)measures.
The proofs are straightforward.
\begin{lem}\label{basix}
\begin{enum}
\item $\hm^\phi\leq\ubox^\phi\leq\pack^\phi$ for any gauge $\phi$.
\item If $\alpha>1$, then $\pack_0^\phi\leq\ubox_0^{\phi^\alpha}$ and
$\pack^\phi\leq\ubox^{\phi^\alpha}$ for any continuous gauge $\phi$.
\item If $\ubox^\phi(X)<s$, then there is a sequence $X_n\upto X$ such that
$\sup\ubox_0^\phi(X_n)<s$; and likewise for $\pack^\phi$.
\end{enum}
\end{lem}

\subsection*{Cartesian products}
Given two metric spaces $X_1$ and $X_2$ with respective metrics
$d_1$ and $d_2$, provide the cartesian product $X_1\times X_2$
with the maximum metric
\[
  d\bigl((x_1,x_2),(y_1,y_2)\bigr)=\max(d_1(x_1,y_1),d_2(x_2,y_2)).
\]

\subsection*{The measures on the Cantor cube}
For $p\in2^{<\Nset}$ we denote $\cyl p=\{x\in\Cset:p\subs x\}$ the cone determined by $p$.
Metrize $\Cset$ by the standard least difference metric: given $x\neq y\in\Cset$,
set $\Delta(x,y)=\min\{i\in\Nset:x(i)\neq y(i)\}$ and define
$d(x,y)=2^{-\Delta(x,y)}$.
The topology induced by $d$ clearly coincides with that of $\Cset$.

Routine proofs show that in this metric, $\hm^1$ coincides on Borel sets with
the usual product measure, i.e., the Haar measure of the compact group $\Cset$,
and that $\hm^1(\Cset)=\ubox^1(\Cset)=1$.

Besides the \si ideals $\NN$ of Haar null sets and $\MM$ of meager sets, we also
make use of the \si ideal $\EE$ generated by closed null sets.
%


\section{Strict measure zero}
\label{sec:sssmz}

Recall that 
Besicovitch~\cite{MR1555386} proved that a set $X$ has strong measure zero
if and only if $\hm^\phi(X)=0$ for every gauge $\phi$ and a similar characterization
of meager-additive sets in terms of fractal measures was provided
in~\cite{MR3946667,MR4472524}. The following definition follows the same pattern.

\begin{defn}
A set $E$ in a metric space is said to have \emph{strict measure zero}
if $\pack^\phi(E)=0$ for every gauge $\phi$. Strict measure zero is abbreviated as \sssmz{}.%
\footnote{Strong measure zero is abbreviated as \smz{}, e.g., in~\cite{MR3453581,MR4146582}
and sharp measure zero as \ssmz{} in~\cite{MR3946667,MR4472524,MR4146582}.}
The family of all \sssmz{} sets in a metric space $X$ is denoted by $\sssmz(X)$.
\end{defn}
It is immediate that $\sssmz(X)$ is a \si ideal.
\begin{prop}\label{unifcont}
A uniformly continuous image of a \sssmz{} set is \sssmz{}.
\end{prop}
This follows easily from the well-known fact that if $f$ is uniformly continuous
and $g$ is its modulus, then $\pack^\phi(f(X))\leq\pack^{\phi{\circ}g}(X)$ for any gauge
$\phi$, see, e.g., \cite[Theorem 29]{MR0281862}.

In this section we establish a few characterizations of \sssmz{} in terms of covers 
and some properties of products.
The following lemma will turn handy. It follows at once from Lemma~\ref{basix}(i) and (ii).
\begin{lem}\label{upsilon}
A set $X$ has \sssmz{} if and only if $\ubox^\phi(X)=0$ for every gauge $\phi$.
\end{lem}

\subsection*{Products}
We provide a few facts about \sssmz{} in cartesian products.
\begin{prop}\label{prodPnull}
If $X$ is \sssmz{}, then
\begin{enum}
\item $\pack^\phi(X\times Y)=0$ whenever $\pack_0^\phi(Y)=0$,
\item $\hm^\phi(X\times Y)=0$ whenever $\hm^\phi(Y)=0$.
\end{enum}
\end{prop}
\begin{proof}
(i)
Suppose that $\pack_0^\phi(Y)=0$. By~\cite[Theorem 4]{MR3209322} there is a gauge $\psi\prec\phi$
such that $\pack_0^\psi(Y)=0$. Let $\xi$ be any doubling gauge satisfying $\xi\geq\frac{\phi}{\psi}$;
such a gauge exists, because $\lim_{r\to0}\frac{\phi(r)}{\psi(r)}=0$.
Since $\xi$ is doubling, it follows from~\cite[Theorem 17]{MR1362951} that
$\pack^{\xi\psi}(X\times Y)\leq\pack^\xi(X)\pack^\psi(Y)$. Since $X$ is \sssmz{}, we have
$\pack^\xi(X)=0$. Therefore
\[
  \pack^\phi(X\times Y)\leq\pack^{\xi\psi}(X\times Y)
  \leq\pack^\xi(X)\pack^\psi(Y)=0\cdot0=0.
\]

(ii) is similar.
Suppose that $\hm^\phi(Y)=0$. By~\cite[Theorem 42]{MR0281862} there is a gauge $\psi\prec\phi$
such that $\hm^\psi(Y)=0$. Let $\xi$ be any gauge satisfying $\xi\geq\frac{\phi}{\psi}$.
By~\cite[Theorem 10]{MR1362951}
$\hm^{\xi\psi}(X\times Y)\leq\pack^\xi(X)\hm^\psi(Y)$. Proceed as above.
\end{proof}
Note that letting $Y$ to be a one point space shows that (i) is actually equivalent with \sssmz{}
and later we will see that if $X\subs\Rset^n$, then (ii) is also equivalent with \sssmz{}.
\begin{thm}\label{products}
Let $X,Y$ be metric spaces.
\begin{enum}
\item If $X$ and $Y$ are \sssmz{}, then $X\times Y$ is \sssmz{}.
\item If $X$ is \sssmz{} and $Y$ is \smz{}, then $X\times Y$ is \smz{}.
\end{enum}
\end{thm}
\begin{proof}
Let $\phi$ be a gauge.

(i) Since $Y$ is \sssmz{}, there is a cover $\{Y_n:n\in\Nset\}$ of $Y$
such that $\pack^\phi_0(Y_n)=0$ for all $n$. By Proposition~\ref{prodPnull}(i)
$\pack^\phi(X\times Y_n)=0$. Hence $\pack^\phi(X\times Y)=0$.

(ii) Since $Y$ is \smz{}, we have $\hm^\phi(Y)=0$. By Proposition~\ref{prodPnull}(ii)
$\hm^\phi(X\times Y)=0$.
\end{proof}

\subsection*{$\gamma$-covers}
Recall that a sequence $\seq{U_n:n\in\Nset}$ of sets in $X$ is called a \emph{$\gamma$-cover}
if each $x\in X$ belongs to all but finitely many $U_n$'s.
The following three combinatorial characterizations of \sssmz{} in terms of $\gamma$-covers
are all alike, but do differ. They have the same pattern as the original Borel's definition
of \smz{}.
\begin{thm}\label{thm:sssmz}
The following are equivalent:
\begin{enum}
\item $X$ has strict measure zero.
\item For each $f\in\UPset$ and for each $\Seqeps$ there is a sequence
$\seq{E_n:n\in\Nset}$ and a partition of $\Nset$ into consecutive intervals
$\seq{I_k:k\in\Nset}$ such that
\begin{enum}
\item[\rm(a)] $\forall n\ \diam E_n\leq\eps_n$,
\item[\rm(b)] $\forall k\ \abs{I_k}\leq f(k)$,
\item[\rm(c)] $\forall x\in X\ \fmany k\ \exists n\in I_k\ x\in E_n$.
\end{enum}
\item For each $\Seqeps$ there is a sequence
$\seq{F_n:n\in\Nset}$ of subsets of $X$ such that
\begin{enum}
\item[\rm(a)] $\forall n\ \abs{F_n}\leq n$,
\item[\rm(b)] $\seq{B(F_n,\eps_n):n\in\Nset}$ is a $\gamma$-cover of $X$.
\end{enum}
\item There is $f\in\UPset$ such that for each $\Seqeps$ there is a sequence
$\seq{E_n:n\in\Nset}$ and a partition of $\Nset$ into consecutive intervals
$\seq{I_k:k\in\Nset}$ such that
\begin{itemize}
\item[\rm(a)] $\forall k\ \forall n\in I_k\ \diam E_n\leq\eps_k$,
\item[\rm(b)] $\forall k\ \abs{I_k}\leq f(k)$,
\item[\rm(c)] $\forall x\in X\ \fmany k\ \exists n\in I_k\ x\in E_n$.
\end{itemize}
\end{enum}
\end{thm}
\begin{proof}
(i)\Implies(ii):
Let $\wh n=f(0)+f(1)+\dots+f(n)$ and $\del_n=\eps_{\wh n}$. Choose a gauge such that
$\phi(\del_n)>1/f(n)$.

Since $X$ is \sssmz{}, we have $\ubox^\phi(X)=0$.
Lemma~\ref{basix}(iii) yields a sequence $X_j\upto X$ such that
$\ubox_0^\phi(X_j)<1$ which in turn yields a sequence $\seq{n_j}$ such that
$$
  \forall j\ \forall n\geq n_j\quad N_{X_n}(\del_n)\cdot\phi(\del_n)<1.
$$
Therefore for all $j$ and $k\in[n_j,n_{j+1})$ there is a $\del_k$-cover $\EEE_k$ of $X_k$
such that (the latter inequality by $\phi(\del_n)>1/f(n)$)
\begin{equation}\label{borel1}
  \abs{\EEE_k}<\frac{1}{\phi(\del_k)}\leq f(k).
\end{equation}
Let $\EEE=\bigcup_{k\in\Nset}\EEE_k$ and enumerate it so that $\EEE_k=\{E_n:n\in I_k\}$
and $I_k$ are consecutive intervals.
So if $n\in I_k$, then
$$
  n\leq\abs{I_0}+\abs{I_1}+\dots+\abs{I_k}\leq f(0)+f(1)+\dots+f(k),
$$
hence $\diam E_n\leq\del_k=\eps_{f(0)+f(1)+\dots+f(k)}\leq\eps_n$ and (a) follows.
Item (b) follows from~\eqref{borel1} and (c) from the fact that $\EEE_k$'s are covers of
$X_k$.

(ii)\Implies(iii): Let $f(n)=n$. Then using (i) choose, for each $j$, $x_j\in E_j$
and let $F_n=\{x_j:j\in I_n\}$.

(iii)\Implies(iv): Start with $\eps_n/2$ instead of $\eps_n$ and consider the sequence
$\seq{F_n:n\in\Nset}$ of (iii). Enumerate the points in $F_1$ by the interval $I_1=[0,\abs{F_1})$,
then $F_2$ by the interval $I_2=[\abs{F_1},\abs{F_1}+\abs{F_2})$ etc.
The resulting sequence is thus $\seq{x_n:n\in\Nset}$.
Let $E_n=B(x_n,\eps_n/2)$ and $f(n)=n$.

(iv)\Implies(i): Let $\phi$ be a gauge. Choose $\eps_k>0$ so that $f(k)\phi(\eps_{k-1})\leq1/k$.
Consider the sequence $\seq{E_k:k\in\Nset}$ from (iv).
Let $Y_k=\bigcup_{n\in I_k}E_n$ and $X_m=\bigcap_{n\geq m}X_n$.
We estimate $\ym^\phi_\eps(X_k)$:
If $\eps_{k}\leq\eps<\eps_{k-1}$, then
$$
  N_{Y_k}(\eps)\phi(\eps)\leq N_{Y_k}(\eps_k)\phi(\eps_{k-1})
  \leq f(k)\phi(\eps_{k-1})\leq1/n,
$$
hence $\ubox^\phi_\eps(X_k)\leq1/k$ and thus $\ubox^\phi_0(X_k)=0$.
Therefore $\ubox^\phi(X_k)=0$. Since (iv)(c) yields $X_k\upto X$, we are done.
\end{proof}

\section{Galvin's game}
\label{sec:galvin}

Galvin~\cite{MR3696064} characterized \smz{} sets in \si compact metric  spaces
in terms of a game. A similar characterization of \ssmz{} sets
(and hence for meager-additive sets) was proved in \cite{MR4472524}.
We consider a similar game and prove a counterpart of Galvin's theorem
for \sssmz{} sets.

\begin{defn}
Let $X$ be a metric space. The game $G^{\musDoubleSharp\mkern-1mu}(X)$ is played as follows:
At the $n$-th inning, Player I chooses $\eps_n>0$
and Player II responds with a finite set $F_n\subs X$ such that
$\abs{F_n}\leq n$.
Player II wins if the sequence of sets $\seq{B(F_n,\eps_n):n\in\Nset}$
is a $\gamma$-cover of $X$, otherwise Player I wins.
\end{defn}

\begin{thm}\label{Galvins}
A metric space $X$ is \sssmz{} if and only if Player I does not have
a winning strategy in $G^{\musDoubleSharp\mkern-1mu}(X)$.
\end{thm}
\begin{proof}
The forward implication:
Suppose that $X$ is \sssmz{}. Then $X$ is a \si totally bounded, so there is an
increasing sequence $X_n\upto X$ of totally bounded sets.

Let $\sigma$ be a strategy for Player I. We show that $\sigma$ is not winning.
Construct by recursion $\del_n>0$ and finite sets $G_n\subs X_n$ subject to the
following conditions:
\begin{enum}
\item[(a)] $\del_n=\min\{\sigma(F_0,F_1,\dots,F_{n-1}):\forall i<n\ F_i\in[G_i]^{\leq i}\}$,
\item[(b)] If $A\subs X_n$ and $\diam A\leq\frac{\del_n}{2}$,
  then $\exists x\in G_n\ A\subs B(x,\del_n)$.
\end{enum}
$G_n$'s exist because $X_n$'s are totally bounded.

Since $X$ is \sssmz{}, by Theorem~\ref{thm:sssmz}(ii) there is a sequence
$\seq{H_n:n\in\Nset}$ such that $H_n\in[X]^n$ for all $n$ and
the sequence $\seq{B(H_n,\frac{\del_n}{4}):n\in\Nset}$ is a $\gamma$-cover of $X$.
For each $y\in H_n$ write $A_y=B(x,\frac{\del_n}{4})\cap X_n$.
 Then $\diam A_y\leq\frac{\del_n}{2}$,
so (b) yields $x\in G_n$ such that $A_y\subs B(x,\del_n)$. In summary,
\begin{equation}\label{galvingame}
  \forall n\ \exists F_n\in[G_n]^{\leq n}\ B(H_n,\tfrac{\del_n}{4})\cap X_n\subs B(F_n,\del_n).
\end{equation}
Let $\eps_n=\sigma(F_0,F_1,\dots,F_{n-1})$. Clearly $\eps_n\geq\del_n$, therefore the sequence
$\seq{\eps_0,F_0,\eps_1,F_1,\dots}$ is played according to the strategy $\sigma$.
Routine verification shows that since both $\seq{B(H_n,\tfrac{\del_n}{4})}$
and $\seq{X_n}$ are $\gamma$-covers, \eqref{galvingame} ensures that $\seq{B(F_n,\del_n)}$
is also a $\gamma$-cover. So Player II wins.
Consequently, $\sigma$ is not a winning strategy for Player I.
We showed that no strategy is winning.

The backwards implication is trivial: if $X$ is not \sssmz{}, then
by Theorem~\ref{thm:sssmz} there is $\Seqeps$ such that $X$ has
no $\seqeps$-fine $\gamma$-cover. The winning strategy for Player I
is of course to play $\eps_n$ at the $n$-th inning.
\end{proof}

\section{Null-additive sets of reals}\label{sec:nadd}

We mentioned it the introduction the Galvin-Mycielski-Solovay Theorem~\cite{GMS}
(abbreviated as GMS Theorem): a set $X\subs\Rset$ has strong measure zero if and only if
for each meager set $M$ the algebraic sum $X+M$ does not cover $\Rset$.
A set $X\subs\Cset$ is called \emph{null-additive} if for each Lebesgue null set $N$
the algebraic sum $X{+}N\in\NN$ is Lebesgue null.
The notion of a null-additive set $X\subs\Rset$ is defined likewise,
and so is that on the circle $\Tset$.
In this section we prove a theorem analogous to GMS Theorem:
a subset of $\Cset$ or $\Rset$ or $\Tset$ is null-additive if and only if it is \sssmz{}.
\begin{thm}\label{mainNN}
Let $\grG=\Cset$ or $\grG=\Rset$ or $\grG=\Tset$. If $X\subs\grG$, the following are equivalent.
\begin{enum}
\item $X$ is \sssmz,
\item $X$ is null-additive,
\item $\hm^1(X\times N)=0$ for each $N\in\NN$.
\end{enum}
\end{thm}
\begin{proof}
We first treat the case of $X\subs\Cset$.

(i)\Implies(iii) follows immediately from Proposition~\ref{prodPnull}(ii), since
the Lebesgue measure on $\Cset$ is $\hm^1$.

(iii)\Implies(ii) follows from the fact that the mapping $(x,y)\mapsto x+y$ is clearly
Lipschitz, hence $\hm^1(X+N)\leq\hm^1(X\times N)=0$

It remains to prove (ii)\Implies(i).
Suppose that $X\subs\Cset$ is null-additive.
Let $\phi$ be a gauge. We verify that $\ubox^\phi(X)=0$.
Choose $F\in\UPset$ such that $F(n)\phi(2^{1-n})\to0$.
Define recursively $f\in\UPset$ subject to
$$
  2^{f(n)}\, (n+1)!\leq
  F(f(n+1)),\quad n\in\Nset.
$$
Clearly
\begin{equation}\label{fuj}
  \forall k\ \forall n<k\quad 2^{f(n)}\, k!\leq F(f(k)).
\end{equation}
At this point we make a heavy use of the Shelah's \cite[Theorem 2.7.18]{MR1350295}:
since $X$ is null-additive, there is a sequence $\seq{H_n:n\in\Nset}$
such that
\begin{enum}
\item[(a)] $\forall n\ H_n\subs 2^{[f(n),f(n+1))}$,
\item[(b)] $\forall n\ \abs{H_n}\leq n$,
\item[(c)] $X\subs\{x\in\Cset:\forall^\infty n\ x\rest [f(n),f(n+1))\in H_n\}$.
\end{enum}
Set
$$
  X_n=\{x\in\Cset:\forall k\geq n\ x\rest[f(k),f(k+1))\in H_k\},\quad n\in\Nset.
$$
We verify that $N_{X_n}(2^{-i})\leq F(i)$
for each $n$ and all $i\geq f(n+1)$. Let $k$ be the unique integer satisfying
$f(k)\leq i<f(k+1)$. In particular, $k>n$. It is obvious that
$$
  N_{X_n}(2^{-i})\leq 2^{f(n)}
  \abs{H_n}\cdot\abs{H_{n+1}}\dots\abs{H_k}\leq
  2^{f(n)}n(n+1)\dots k\leq 2^{f(n)}k!.
$$
Therefore~\eqref{fuj} yields $N_{X_n}(2^{-i})\leq F(f(k))\leq F(i)$ and hence
\[
  \ubox_0^\phi(X_k)=\varlimsup_{r\to0}N_{X_k}(r)\cdot\phi(r)
  \leq\varlimsup_{i\to\infty}F(i)(2^{-i})\phi(2^{1-i})=0.
\]
By (c) $\{X_n:n\in\Nset\}$ is a cover of $X$. Therefore
$\ubox^\phi(X)=0$. Since this holds for any gauge $\phi$, it follows that $X$ has \sssmz{}.
We proved (ii)\Implies(i) for $X\subs\Cset$.

We now show the theorem for $X\subs\Rset$. (i)\Implies(iii)\Implies(ii) is proved exactly as above.
We show (ii)\Implies (i).

Since all properties in question are \si additive, we may suppose that $X\subs[0,1)$.
By~\cite[Theorem 13]{MR3241126} there is a set $A\subs\Cset$
and a Lipschitz surjection $H:A\to[0,1)$ such that preimages of null-additive sets
are null-additive. So supposing $X\subs[0,1)$ is null-additive,
it follows that $H^{-1}(X)\subs\Cset$ is null-additive and thus by the above
is \sssmz{}. Since $H$ is Lipschitz, $X=H(H^{-1}(X))$ is also
\sssmz{}. 

If $X\subs\Tset$, We just need to show that if $X$ is null-additive in $\Tset$, then it is null-additive
in $\Rset$, which is straightforward.
\end{proof}
\begin{coro}
Let $T\colon\Cset\to[0,1]$ be the canonical bijection given by
$T(x)=\sum_{n\in\Nset}2^{-n-1}x(n)$.
Then $X\subs\Cset$ is null-additive if and only if $T(X)$ is null-additive.
\end{coro}
\begin{proof}
If $X\subs\Cset$ is null-additive, then it is \sssmz. Since $T$ is uniformly continuous,
$T(X)\subs[0,1]$ is also \sssmz{} and hence null-additive.

In the other direction, if $Y\subs[0,1]$ is null-additive, then it is \sssmz.
As shown in~\cite{invariants2}, for every set $E\subs[0,1]$ the preimage $T^{-1}(E)$ is covered by
three sets whose diameters are no more than that of $E$. Therefore
$N_{T^{-1}(Y)}(\del)\leq 3N_{Y}(\del)$ for each $\del>0$ and that is enough for
$\ubox^\phi(T^{-1}(Y))\leq 3\ubox^\phi(Y)$ for each gauge $\phi$, and that in turn suffices
to conclude that since $Y$ is \sssmz, so is $T^{-1}(Y)$ and a fortiori
$T^{-1}(Y)$ is null-additive.
\end{proof}
%

\section{Null-additive sets in locally compact Polish groups}
\label{sec:nadditive}

The notion of null-additive sets easily extends to locally compact Polish metric groups.
In this section we look at the relationship of strict measure zero and null-additivity
in these groups. Let us first notice that  \sssmz{} in these groups is independent of
a particular metric. The routine proof is skipped.
\begin{prop}
Let $\grG$ be a locally compact metric group. Then $X\subs\grG$ has \sssmz{}
if and only if for any sequence $\seq{G_n:n\in\Nset}$ of neighborhoods of $1_\grG$ there is
a sequence $\seq{F_n:n\in\Nset}$ of sets in $\grG$ such that $\abs{F_n}\leq n$
and $\seq{F_nG_n:n\in\Nset}$ is a $\gamma$-cover of $X$.

In particular, whether $X$ has \sssmz{} does not depend on the metric on $\grG$.
\end{prop}

\begin{defn}
Let $\grG$ be a locally compact separable metric group. Denote by $\mu_\grG$ its left Haar measure
and by $\NN(\grG)$ the \si ideal of Haar null sets.
A set $X\subs\grG$ is called
\emph{null-additive} if $XN\in\NN(\grG)$ for each $N\in\NN(\grG)$,
The family of all null-additive sets in $\grG$ is denoted by $\NA(\grG)$.
\end{defn}
The proof of the main theorem utilizes \emph{weighted Hausdorff measures}
introduced in~\cite[2.10.24]{MR0257325}: let $X$ be a metric space and $E\subs X$.
Consider a family of pairs $\EEE=\{(E_i,w_i):i\in I\}\subs\mc P(X)\times(0,\infty)$.
Its multiplicity at $x\in X$ is defined by $\mult(\EEE,x)=\sum\{w_i:x\in E_i\}$.
The family $\EEE$ is called
\begin{itemyze}
\item a \emph{weighted cover} of $E$ if $\mult(\EEE,x)\geq1$ for all $x\in E$,
\item a \emph{large weighted cover} of $E$ if $\mult(\EEE,x)=\infty$ for all $x\in E$
\end{itemyze}
and if $\diam E_i\leq\del$ for all $i\in I$, then it is termed a
(large) weighted $\del$-cover.
\begin{defn}[{\cite{MR0257325}}]
For a gauge $\phi$ and $\del>0$ let
\[
  \whm_\del^\phi(E)=\inf\Bigl\{\sum_{i\in\omega}w_i\phi(\diam E_i):
  \{(E_i,w_i)\}\text{ is a weighted $\del$-cover of $E$}\Bigr\}.
\]
The \emph{weighted Hausdorff measure induced by $\phi$} is defined by
\[
  \whm^\phi(E)=\sup_{\del>0}\whm_\del^\phi(E).
\]
\end{defn}
It is obvious that $\whm^\phi\leq\hm^\phi$. It is less obvious that if $\phi$ is doubling,
then $\whm^\phi=\hm^\phi$, see \cite{MR0257325,MR1317515}.
We will need a simple lemma that parallels~\cite[Theorem 42]{MR0281862}.
\begin{lem}\label{largweight}
$\whm^\phi(X)=0$ if and only if there is a large weighted cover $\{(E_i,w_i):i\in\Nset\}$
of $X$ such that $\sum_{i\in\Nset} w_i\phi(\diam E_i)<\infty$.
\end{lem}
\begin{proof}
$\Rightarrow$ Suppose $\whm^\phi(X)=0$. For each $n\in\Nset$ there is a weighted $2^{-n}$-cover
$\{(E_i^n,w_i^n):i\in\Nset\}$ such that $\sum_i w_i^n\phi(\diam E_i^n)\leq2^{-n}$.
The weighted cover $\{(E_i^n,w_i^n):n,i\in\Nset\}$ has obviously the required properties.

$\Leftarrow$ Suppose $\EEE=\{(E_i,w_i):i\in\Nset\}$ is a large weighted cover of $X$
such that $\sum_i w_i\phi(\diam E_i)<\infty$.
Let $\del>0$. There is $n$ such that $\sum_{i\geq n}w_i\phi(\diam E_i)<\del$.
Consider the family $\EEE'=\{(E_i,w_i):i\geq n, \diam E_i\leq\del\}$.
It is a weighted cover. Indeed, for each $x\in X$
\begin{align*}
  &\sum\{w_i:x\in E_i, (E_i,c_i)\in\EEE'\}=\sum\{w_i:x\in E_i\}\\
  &-\left(\sum\{w_i:x\in E_i,i<n\}+\sum\{w_i:x\in E_i,dE_i>\del\}\right)
\end{align*}
and since the first term is infinite and the other two are finite, the multiplicity of
$\EEE'$ is $\infty$, and \emph{a fortiori} $\geq1$, at each $x$.
\end{proof}
\begin{thm}\label{hmwhm}
Let $\grG$ be a locally compact Polish group equipped with a left-invariant metric,
$X\subs\grG$ a \sssmz{} set and $N\subs\grG$.
\begin{enum}
\item If $\hm^\phi(E)=0$, then $\hm^\phi(XE)=0$,
\item If $\whm^\phi(E)=0$, then $\whm^\phi(XE)=0$.
\end{enum}
\end{thm}
\begin{proof}
Let $E\subs\grG$ and suppose that $\whm^\phi(E)=0$.
By Lemma~\ref{largweight} there is a large weighted cover $\EEE=\{(E_n,w_n):n\in\Nset\}$
of $E$ such that $\sum_n w_n\phi(\diam E_n)<\infty$.
Omitting finitely many terms we may suppose all $E_n$ be compact.
It is easy to prove that there is $f\in\UPset$ such that
$\sum_n w_n\phi(\diam E_n)f(n)<\infty$.
Since $\phi$ is right-continuous,
there is $\Seqeps$ such that for all $n$
\[
  \phi(\diam E_n+\eps_n)<\phi(\diam E_n)+\frac{2^{-n}}{w_nf(n)}.
\]
\begin{Claim}\label{Claim}
$\forall n\ \exists\del_n>0 \ \diam C\leq\del_n\Rightarrow\diam(CE_n)\leq\diam E_n+\eps_n$.
\end{Claim}
\begin{cproof}
The right shift is continuous, so since $E_n$ is compact, there is $\del_n>0$ such that
$d(x,x')\leq\del_n\Rightarrow\forall y\in E_n\ d(xy,x'y)\leq\eps_n$.
Suppose $C\subs\grG$ and $\diam C\leq\del$.
Let $x,x'\in C$ and $y,y'\in E_n$. Then
\[
  d(xy,x'y')\leq d(xy,x'y)+d(x'y,x'y')\leq\eps+d(y,y')\leq\eps+\diam E_n
\]
and $\diam(CE_n)\leq\diam E_n+\eps$ follows.
\end{cproof}
Consider the sequence $\seq{\del_n}$.
Since $X$ is \sssmz{}, Theorem~\ref{thm:sssmz} yields a sequence $\seq{\CC_n:n\in\Nset}$
of families of sets such that
\begin{enum}
\item[\rm(a)] $\forall n\in\Nset\ \abs{\CC_n}\leq f(n)$,
\item[\rm(b)] $\forall n\in\Nset\ \forall C\in\CC_n\ \diam C<\del_n$,
\item[\rm(c)] $\forall x\in X\ \fmany n\in\Nset\ \exists C\in\CC_n\ x\in C$.
\end{enum}
Let $\AAA=\{(CE_n, c_n):C\in\CC_n,\ n\in\Nset\}$. We show that $\AAA$ witnesses $\whm^\phi(XE)=0$.
Let $z\in XE$. There are $x\in X$, $y\in N$ be such that $z=xy$.
Then
\begin{enum}
\item[\rm(d)] $\sum\{w_n:x\in E_n\}=\infty$,
\item[\rm(e)] $\exists k\ \forall n\geq k\ \exists C\in\CC_n\ y\in C$.
\end{enum}
Therefore
$
  \mult(\AAA,z)\geq\mult(\EEE,x)-\sum\{w_n:n<k\}=\infty,
$
so $\AAA$ is a large weighted cover of $XN$. Moreover
\begin{align*}
  \sum_{(A,w)\in\AAA}w\phi(\diam A)&=\sum_{n\in\Nset}\sum_{C\in\CC_n}w_n\phi(\diam(CE_n)) \\
  &\leq \sum_{n\in\Nset}w_n\phi(\diam E_n+\eps_n)\cdot\abs{\CC_n} \\
  &\leq \sum_{n\in\Nset}w_n\Bigl(\phi(\diam E_n)+\frac{2^{-n}}{w_nf(n)}\Bigr)\cdot f(n) \\
  &= \sum_{n\in\Nset}w_n\phi(\diam E_n)+2^{-n}<\infty,
\end{align*}
so Lemma~\ref{largweight} yields $\whm^\phi(XE)=0$.
We proved (ii). The proof of (i) is identical, we just have to assume that all weights $w_i=1$.
\end{proof}
We are arriving to one of the main theorems of the paper.
\begin{thm}\label{main1}
If $\grG$ is a locally compact Polish group, then $\sssmz(\grG)\subs\NA(\grG)$.
In other words, every \sssmz{} set in $\grG$ is null-additive.
\end{thm}
\begin{proof}
The proof requires a structural description of the Haar measure that can be traced down
to the Bandt's paper~\cite{MR677233}.
By Struble's theorem~\cite{MR348037} every locally compact metric group has a left-invariant
metric such that every closed set with a finite diameter is compact.
Provide $\grG$ with such a metric $d$ (let's call is a \emph{Struble metric})
and define a function
\begin{equation}\label{bandt1}
  \phi(r)=\sup\{\mu_\grG(E):\diam E\leq r\},\ r\geq 0.
\end{equation}
Clearly, $\phi$ is nondecreasing and takes only finite values. It is right-continuous:
otherwise there is $r,\eps>0$ and a sequence
of sets $E_n$ such that $\diam E_n{\searrow} r$ and $\mu(E_n)\geq\phi(r)+\eps$.
By left-shifting $E_n$'s we may suppose that there is compact set $C$ containing all of them.
Using Blaschke Selection Theorem, choose a subsequence that converges in Hausdorff metric to
a set $E\subs C$.
Then $E=\bigcap_{n\in\Nset}\overline{\bigcup_{m\geq n}E_n}$, so $\mu(E)\geq\phi(r)+\eps>0$ and
at the same time $\diam E\leq r$: a contradiction.
So $\phi$ is a gauge. Therefore it makes sense to consider the measure $\whm^\phi$.
The following surprising link of Haar measure and $\whm^\phi$ can be extracted from the
proofs in~\cite{MR677233}:
\begin{equation}\label{bandt2}
  \whm^\phi=\mu_\grG.
\end{equation}
It remains to apply Theorem~\ref{hmwhm}(ii).
\end{proof}

\begin{defn}
A locally compact Polish group $\grG$ is called a \GMS{} group if $\NA(\grG)\subs\sssmz(\grG)$,
or, equivalently, if $\NA(\grG)=\sssmz(\grG)$.%
\footnote{In~\cite{MR3453581}, a Polish group $\grG$ is called a $\mathsf{GMS}$ group
if it satisfies the Galvin-Mycielski-Solovay Theorem, i.e., a set
$X\subs\grG$ has strong measure zero if and only if $XM\neq\grG$ for all meager sets $M$.}
\end{defn}
\begin{thm}\label{main2}
A product of finitely many \GMS{} groups is a \GMS{} group.
\end{thm}
\begin{proof}
Let $\grG$ and $\grH$ be \GMS{} groups and $X\subs\grG\times\grH$ a null-additive set.
Let $X_\grG$ be the projection of $X$ on $\grG$.
If $N\subs\grG$ is a Haar null set, then $N\times\grH$ is a Haar null set in $\grG\times\grH$.
Therefore $X\cdot(N\times\grH)$ is a Haar null set in $\grG\times\grH$.
Since $X\cdot(N\times\grH)=(X_\grG\cdot N)\times\grH$,
it follows from Fubini Theorem that $X_\grG\cdot N$ is Haar null in $\grG$.
We conclude that $X_\grG$ is null-additive.
By the same argument, the projection $X_\grH$ of $X$ on $\grH$ is also null-additive.

Therefore, since both $\grG$ and $\grH$ are \GMS{} groups, it follows that both
$X_\grG$ and $X_\grH$ are \sssmz{} sets.
By Theorem~\ref{products}(i), $X_\grG\times X_\grH$ is a \sssmz{} set and
since $X\subs X_\grG\times X_\grH$, so is $X$.
We proved that every null-additive set in $\grG\times\grH$ is \sssmz.
The reverse inclusion follows from Theorem~\ref{main1}.
\end{proof}
The following theorem is a trivial consequence of Theorems~\ref{mainNN} and~\ref{main2}.
\begin{thm}\label{euclid}
$\Cset,\Rset^n,\Tset^n$ are \GMS{} groups for all $n$.
\end{thm}
Shelah~\cite{MR1324470} proved that in $\Cset$ every null-additive set meager-additive.
This generalizes to \GMS{} groups:
\begin{coro}
A null-additive set in a \GMS{} group is meager-additive.
\end{coro}
\begin{proof}
A null-additive set in a \GMS{} group is \sssmz{} and trivially every \sssmz{} set is \ssmz{}.
By~\cite[Theorem 4.1]{MR4472524}, every \ssmz{} set in a locally compact group is meager-additive.
\end{proof}
\begin{coro}\label{null1}
Let $\grG,\grH$ be locally compact Polish groups.
\begin{enum}
\item If $\grG,\grH$ are \GMS{} and $X\subs\grG$, $Y\subs\grH$ are null-additive,
then so is $X\times Y$.
\item If $\grG$ is \GMS{}, $X\subs\grG$ is null-additive and $f:\grG\to\grH$ is continuous,
then $f(X)$ is null-additive.
\end{enum}
\end{coro}
\begin{proof}
(i) is immediate from Theorem~\ref{products}(i) and
(ii) from Proposition~\ref{unifcont} and Theorem~\ref{main1}.
\end{proof}
\begin{coro}\label{null2}
Let $\grG$ be a \GMS group and $X,N\subs\grG$. If $X$ is null-additive, then
\begin{enum}
\item If $\hm^\phi(N)=0$, then $\hm^\phi(XN)=0$,
\item $\hdim(XN)=\hdim N$.
\end{enum}
\end{coro}
\begin{proof}
This follows immediately from Theorem~\ref{hmwhm}.
\end{proof}
\begin{coro}\label{null3}
Let $X\subs\Rset^2$. The following are equivalent.
\begin{enum}
\item $X$ is null-additive,
\item all orthogonal projections of $X$ on lines are null-additive,
\item at least two orthogonal projections of $X$ on lines are null-additive.
\end{enum}
\end{coro}
\begin{proof}
(i)\Implies(ii) Since $X$ is null-additive, it is \ssmz. Therefore projections,
being uniformly continuous, are \sssmz, hence null-additive.

(iii)\Implies(i) We may assume that the two projections are on the coordinate lines.
Their product is null-additive by Corollary~\ref{null1}.
Hence $X$, being a subset of their product, is null-additive.
\end{proof}

\section{Remarks and questions}\label{sec:rem}
\subsection*{\GMS{} groups}
The paramount question is, of course, if Theorem~\ref{main1} has a converse.
\begin{question}
Which locally compact Polish groups are \GMS{}? 
In particular, is every null-additive set in a locally compact Polish group \sssmz{}?
\end{question}

\subsection*{$\EE$-additive sets}

We look at the ideal of subsets of a locally compact Polish group $\grG$ generated by \si compact
sets of Haar measure zero. It is denoted by $\EE(\grG)$. We also consider the ideal of meager
subsets of $\grG$ denoted by $\MM(\grG)$. In case of $\grG=\Cset$ we write $\EE$ and $\MM$.
\begin{defn}
If $\mc J$ is any of the discussed ideals, a set $X$ in a locally compact group $\grG$
is called \emph{$\mc J$-additive} if $XJ\in\mc J$ for each $J\in\mc J$ It is denoted
$\mc J\mkern-4mu\mc A$. Meager-additive means $\MM$-additive.
\end{defn}
It was proved in \cite{MR3946667} that $\MA=\EA$. Though we are not able to extend this
equivalence to more general groups, we can utilize techniques used in the proofs of
Theorems~\ref{hmwhm} and~\ref{main1} to prove that
in a locally compact Polish group admitting an invariant metric
every $\MM$-additive set is $\EE$-additive, answering thus~\cite[Question 8.4(i)]{MR4472524}. 
We will need a lemma on weighted Haudorff measure on compact sets.
\begin{lem}\label{wcpt}
If $E$ is compact and $\whm^\phi(E)<s$, then for any $\del>0$ there is a finite weighted
$\del$-cover $\{(E_i,w_i)\}$ such that $\sum_iw_i\phi(\diam E_i)<s$.
\end{lem}
\begin{proof}
Let $\eps<\del$ and let $\EEE$ be a weighted $\eps$-cover such that
$\sum_{(E,w)\in\EEE}w\phi(\diam E)<s$. Inflating slightly the sets and the weights yields
an open weighted $\del$-cover $\AAA$ such that $\sum_{(E,w)\in\AAA}w\phi(\diam E)$
is still below $s$ and the multiplicities of $\AAA$ are strictly more than $1$
at each point of $E$.
This lets us choose for each $x\in E$ a finite family $\AAA_x\subs\AAA$
such that $\mult(\AAA_x,x)\geq1$.
Set $U_x=\bigcap\{E:(E,w)\in\AAA_x\}$.
The family $\{U_x:x\in E\}$ is obviously an open cover of $E$ and hence there is
a finite set $F\subs E$ such that $\{U_x:x\in F\}$ is a cover of $E$.
Therefore the family $\{(E,w)\in\AAA:E\in \AAA_x,x\in F\}$ is the weighted $\del$-cover
of $E$ we are looking for.
\end{proof}
\begin{thm}\label{MMEE}
Let $\grG$ be a locally compact Polish group admitting an invariant metric.
Every $\MM$-additive set in $\grG$ is $\EE$-additive.
\end{thm}
\begin{proof}
Let $E\in\EE$. We may assume that $E$ is compact.
Let $d$ be a Strobel metric on $\grG$. Let $\phi$ be the gauge given by~\eqref{bandt1}.
By~\eqref{bandt2} $\whm^\phi(E)=0$. By Lemma~\ref{wcpt} there is, for each $n\in\Nset$,
a finite weighted cover $\AAA_n=\{(A_n^k,w_n^k):k\in I_n\}$ of $E$ such that
$\sum_{k\in I_n}w_n^k\phi(\diam A_n^k)<2^{-n}$. We may assume that all sets $A_n^k$ are compact.
For each $n$ find $\eps_n>0$ such that for all $k\in I_n$
\[
  \phi(\diam A^k_n+\eps_n)<\phi(\diam A^k_n)+\frac{2^{-n}}{w^k_n\abs{I_n}}.
\]
It is possible because $\phi$ is right-continuous and $I_n$ is finite.
Then use Claim~\ref{Claim} to get a sequence of $\del_n>0$ such that for all $n\in\Nset$
and $k\in I_n$
\[
  \diam C\leq\del_n\Rightarrow \diam(CA^k_n)\leq\diam A^k_n+\eps_n.
\]
Since $X$ is $\MM$-additive, it has by~\cite[Theorem 5.1]{MR4472524} \emph{sharp measure zero}, 
which means the following: for any sequence $\seq{\del_n:n\in\Nset}$ of positive reals
there is a sequence of sets $\seq{C_n:n\in\Nset}$ such that $\forall n\ \diam C_n<\del_n$ 
(and hence Strobel metric guarantees that $C_n$'s are compact)
with the following property: there is a partition $\{D_j:j\in\Nset\}$ of $\Nset$ into finite sets
such that $\forall x\in X\ \fmany j\ \exists n\in I_j\ x\in C_n$.
The set 
\begin{equation}\label{Xhat}
  \wh X=\{x\in\grG:\fmany j\ \exists n\in D_j\ x\in C_n\}
\end{equation}
is clearly \si compact and $X\subs\wh X$. Consider the family
\[
  \AAA=\{(C_nA_n^k,w_n^k):n\in\omega,k\in I_n\}.
\]
Let $z\in \wh XE$. There are $x\in\wh X$, $y\in E$ such that $z=xy$.
As to $x$, by \eqref{Xhat} there is an infinite set $J\subs\Nset$ such that 
$\forall n\in J\ x\in C_n$.
We also have $\forall n\in J\ \mult(\AAA_n,y)\geq1$, hence
\[
  \mult(\AAA,z)\geq\sum_{n\in J}\mult(\AAA_n,y)\geq\sum_{n\in J}1=\abs J=\infty,
\]
so $\AAA$ is a large cover of $\wh XE$.
Moreover
\begin{align*}
  \sum_{(A,c)\in\AAA}c\phi(\diam A)
  &\leq\sum_{n\in\Nset}\sum_{k\in I_n}c_n^k \phi(\diam(C_nA^k_n))\\
  &\leq\sum_{n\in\Nset}\sum_{k\in I_n}c^k_n\Bigl(\phi(\diam A^k_n)+\frac{2^{-n}}
    {c^k_n\abs{I_n}}\Bigr)\\
  &\leq\sum_{n\in\Nset}\Bigl(\sum_{k\in I_n}c^k_n\phi(\diam A^k_n)+2^{-n}\Bigr)
  \leq\sum_{n\in\Nset}2\cdot2^{-n}<\infty
\end{align*}
so Lemma~\ref{largweight} yields $\whm^\phi(\wh XE)=0$. Since $\wh XE$, being
a continuous image of the product $\wh X\times E$ of \si compact sets, is \si compact,
it follows that $XE\subs\wh XE\in\EE$.
\end{proof}
The other inclusion remains a mystery, except for $\Cset$ and Euclidean spaces, cf.~\cite{MR3946667}.
\begin{question}[{\cite[Question 8.4(ii)]{MR4472524}}]
Is every $\EE$-additive set is a locally compact group $\MM$-additive?
\end{question}

\subsection*{Sharply additive sets}
A set $X\subs\grG$ in a locally compact group is called \emph{sharply $\MM$-additive} if 
for every meager set $M\subs\grG$ there is a \si compact set $C\supseteq X$ such that $CM$ is meager.
This notion comes from~\cite{MR4472524} where it was proved that a set in a locally compact group is
sharply $\MM$-additive if and only if it is \ssmz{} and also that 
a set in a locally compact group that admits an invariant metric is sharply $\MM$-additive
if and only if it is $\MM$-additive.

The notions of \emph{sharply $\NN$-additive} and \emph{sharply $\EE$-additive} are defined likewise.
The proofs of Theorems~\ref{main1} and ~\ref{MMEE} actually show that
for a set $X\subs\grG$ in a locally compact Polish group,
\begin{itemyze}
\item If $X$ is \sssmz{}, then $X$ is sharply $\NN$-additive,
\item If $X$ is \ssmz{}, then $X$ is sharply $\EE$-additive.
\end{itemyze}
It follows, e.g.,  that if every sharply $\NN$-additive set in $\grG$ is \sssmz{}, then $\grG$ is \GMS{}.
\begin{questions}
\begin{enum}
\item Is every $\NN$-additive set in a locally compact Polish group sharply $\NN$-additive?
\item Is every $\EE$-additive set in a locally compact Polish group sharply $\EE$-additive?
\end{enum}
\end{questions}
\subsection*{Uniformities and additivities of \sssmz{} and null-additive sets}
We mention results regarding cardinal invariants of strict measure zero and null-additive sets from~\cite{invariants2}.
Recall that the \emph{uniformity} and \emph{additivity} of an ideal $\mc J$ of sets are,
respectively, the cardinal numbers
\begin{align*}
  \non\mc J&=\min\{\abs A:A\subs X\land A\notin \mc J\}, \\
  \add\mc J&=\min\{\abs{\mc A}:\mc A\subs\mc J\land\bigcup\mc A\notin \mc J\}.
\end{align*}
Note that $\non\NA$ is often referred to as transitive additivity of $\NN$, and likewise $\non\MA$.
\begin{thm}[{\cite{invariants2}}] Let $X$ be an uncountable analytic metric space.
\begin{enum}
\item If $X$ \si totally bounded, then $\add\sssmz(X)=\non\sssmz(X)=\non\NA$.
\item If $X$ contains a uniform copy of $\Pset$, then $\add\sssmz(X)=\non\sssmz(X)=\addN$.
\end{enum}
\end{thm}
\begin{coro}[{\cite{invariants2}}]
Let $\grG$ be a locally compact group. 
\begin{enum}
\item $\non\NA\leq\non\NA(\grG)$,
\item if $\grG$ is \GMS{}, then  $\non\NA(\grG)=\non\NA$.
\end{enum}
\end{coro}
It follows that consistently \sssmz{} is not a topological invariant: Pawlikowski~\cite{MR776210}
has a model in which $\add\NN<\non\NA$. Choose a $X\subs\Pset$ that is not \sssmz{} and
$\abs X=\add\NN$. Let $e:\Pset\to\Rset$ be a homeomorphic embedding. 
Then $\abs{e(X)}<\non\NA$ and therefore $e(X)$ is \sssmz{} and homeomorphic to $X$.
\begin{question}
Is $\non\NA(\grG)=\non\NA$ for all locally compact Polish groups?
\end{question}
\subsection*{ $\gamma$-sets and strong $\gamma$-sets}
A family of open sets $\UU$ in a metric space $X$ is an $\omega$-cover if
for each finite set $F\subs X$ there is $U\in\UU$ such that $F\subs U$,
and an \emph{$n$-cover} if for each set $F\subs X$
of cardinality $n$ there is $U\in\UU$ such that $F\subs U$.

A set $X$ in a metric space is called a \emph{$\gamma$-set} if every $\omega$-cover of $X$ 
contains a $\gamma$-cover. The notion comes from~\cite{MR667661}. 
There are several models exhibiting consistency of a $\gamma$-set that is not  null-additive,
e.g., Szewczak and Weiss~\cite[Theorem 3.3]{MR4472532}. Conversely, $\mathfrak{p}<\add\NN$
yields a set that is null-additive, yet not a $\gamma$-set.

The situation differs when we strenghen $\gamma$ to strong $\gamma$: 
a set $X$ in a metric space is called a \emph{strong $\gamma$-set} if for any sequence
$\seq{\UU_n:n\in\Nset}$ where $\UU_n$ is an $n$-open cover of $X$ there is a $\gamma$-cover
$\seq{U_n:n\in\Nset}$ with $U_n\in\UU_n$. The notion comes from
Galvin and Miller~\cite{MR738943} (but see also \cite{MR2193331}). They show
that every strong $\gamma$-set of reals is null-additive. Their result follows immediately
from the following theorem.
\begin{thm}
\begin{enum}
\item Every strong $\gamma$-set in a metric space has \sssmz{}.
\item Every strong $\gamma$-set in a locally compact Polish group is null-additive.
\end{enum}
\end{thm}
\begin{proof}
Let $X$ be a strong $\gamma$-set.
(i) For each $n\in\Nset$ let $\eps_n>0$ and $\UU_n=\{B(F,\eps_n):F\subs X\land\abs F=n\}$.
Since $\UU_n$ is a cover $[X]^n$, there is a $\gamma$-cover $\seq{U_n:n\in\Nset}$ with $U_n\in\UU_n$.
Apply Theorem~\ref{thm:sssmz}(iii).

(ii) follows from (i) and Theorem~\ref{main1}.
\end{proof}
Since, as shown above, there is consistently a $\gamma$-set that is not null-additive,
the converses of both (i) and (ii) fail.
\bibliographystyle{amsplain}
\bibliography{everything}
\end{document}